\documentclass[a4paper, 11pt]{amsart}

\usepackage{dsfont}
\usepackage{amsmath, amsthm, amssymb}
\usepackage{fullpage}
\usepackage{xypic}
\usepackage{graphicx}
\usepackage{color}
\usepackage{mathtools}
\usepackage[usenames,dvipsnames]{xcolor}
\usepackage[latin1]{inputenc}
\usepackage{indentfirst}
\usepackage{epstopdf}
\usepackage{subfigure}
\usepackage{emptypage}
\usepackage{hyperref}
\hypersetup{
    colorlinks=true, 
    linktoc=all,     
    linkcolor=blue,  
    citecolor=JungleGreen,
}

\newtheorem{theorem}{Theorem}

\newtheorem{corollary}[theorem]{Corollary}

\newtheorem{proposition}[theorem]{Proposition}
\newtheorem{lemma}[theorem]{Lemma}
\newtheorem{definition}[theorem]{Definition}
\newtheorem{remark}[theorem]{Remark}

\newtheorem{theorem*}{Theorem}
\newtheorem{question*}[theorem*]{Question}
\newtheorem{conjecture*}[theorem*]{Conjecture}
\newtheorem{corollary*}[theorem*]{Corollary}
\newtheorem{theorem*e}{Teorema}
\newtheorem{question*e}[theorem*e]{Pregunta}
\newtheorem{conjecture*e}[theorem*e]{Conjetura}
\newtheorem{corollary*e}[theorem*e]{Corolario}

\renewcommand{\ker}{\mathrm{ker}}
\newcommand{\im}{\mathrm{im}}

\newcommand{\id}{\mathrm{id}}

\newcommand{\R}{\mathbb{R}}

\newcommand{\fcont}{\mathfrak{FCont}}
\newcommand{\Cont}{\mathfrak{Cont}}

\newcommand{\pr}{\mathrm{pr}}

\begin{document}

\title{Tight neighborhoods of contact submanifolds}

\author{Luis Hern\'{a}ndez--Corbato}
\address{Instituto de Ciencias Matematicas CSIC--UAM--UCM--UC3M, C. Nicol\'{a}s Cabrera, 13--15, 28049, Madrid, Spain}
\email{luishcorbato@mat.ucm.es}

\author{Luc\'{i}a Mart\'{i}n--Merch\'{a}n}
\address{Departamento de \'Algebra, Geometr\'{\i}a y Topolog\'{\i}a\\ Universidad Complutense de Madrid \\ Plaza de Ciencias 3 \\ 28040 Madrid \\ Spain}
\email{lmmerchan@ucm.es}

\author{Francisco Presas}
\address{Instituto de Ciencias Matematicas CSIC--UAM--UCM--UC3M, C. Nicol\'{a}s Cabrera, 13--15, 28049, Madrid, Spain}
\email{fpresas@icmat.es}



\begin{abstract}
We prove that any small enough neighborhood of a closed contact submanifold is always tight under a mild assumption on its normal bundle. The non-existence of $C^0$--small positive loops of contactomorphisms in general overtwisted manifolds is shown as a corollary.
\end{abstract}

\maketitle

\section{Introduction}

A contact manifold $(M, \xi)$ is an $(2n+1)$--dimensional manifold equipped with a maximally non--integrable codimension 1 distribution $\xi \subset TM$. If we assume that $\xi$ is coorientable, as will be the case in the article, the hyperplane distribution can be written as the kernel of a global 1--form $\alpha$, $\xi = \ker(\alpha)$, and the maximal non--integrable condition reads as $\alpha \wedge (d\alpha)^{n} \neq 0$. These conditions imply that $(\xi, d\alpha)$ is a symplectic vector bundle over $M$. However, a contact structure on $M$ cannot be directly recovered from a hyperplane distribution $\xi$ and a symplectic structure $\omega$ on the fibers. The formal data $(\xi, \omega)$ is called formal contact structure.

Let $\Cont(M)$ and $\fcont(M)$ denote the set of contact and formal contact structures, respectively. Gromov proved that if $M$ is open the natural inclusion is a homotopy equivalence. The statement does not readily extend to closed manifolds. In dimension 3, Eliashberg introduced a subclass $\Cont_{OT}(M)$ of $\Cont(M)$, the so--called overtwisted contact structures, and proved that any formal contact homotopy class contains a unique, up to isotopy, overtwisted contact structure. Recently, this result has been extended to arbitrary dimension in \cite{BEM} so the notion of overtwisted contact structure has been settled in general.

Prior to \cite{BEM}, different proposals for the definition of the overtwisting phenomenum appeared in the literature. The plastikstufe, introduced in \cite{niederkruger}, resembled the overtwisted disk in the sense that it provides an obstruction to symplectic fillability. The presence of a plastikstufe has been shown to be equivalent to the contact structure being overtwisted (check \cite[Theorem 1.1]{CMP} and \cite{huang} for a list of disguises of an overtwisted structure). One of the corollaries obtained in \cite{CMP} is a stability property for overtwisted structures: if $(M, \ker\, \alpha)$ is overtwisted then $(M \times \mathbb{D}^2(R), \ker(\alpha + r^2 d\theta))$ is also overtwisted provided $R > 0$ is large enough, where $\mathbb{D}^2(R)$ denotes the open 2--disk of radius $R$ and $r^2d\theta$ denotes the standard radial Liouville form in $\mathbb{R}^2$.

\subsection{Statements of the results}

This paper explores the other end of the previous discussion, can small neighborhoods of contact submanifolds be overtwisted? We provide a negative answer to the question in several instances. The main result presented in the article is the following:

\begin{theorem}\label{thm:MD2tight}
Let $(M, \ker \, \alpha)$ be a contact manifold. Then there exists $\varepsilon > 0$ such that $(M \times \mathbb{D}^2(\varepsilon), \ker(\alpha + r^2\, d\theta))$ is tight.
\end{theorem}

This theorem was previously obtained by Gironella \cite[Corollary H]{gironella} in the case of 3--manifolds with a completely different approach. An interesting consequence is stated in the next corollary:


\begin{corollary}
Given any overtwisted contact manifold $(M, \alpha)$, there exists a radius $R_0 \in \R^+ \setminus \{0\}$ such that $(M \times \mathbb{D}^2(R), \alpha + r^2 \, d\theta)$ is tight if $R \in (0, R_0)$ and is overtwisted if $R > R_0$.
\end{corollary}

Note that a similar statement was already proven in \cite{niederkrugerpresas} but in the case of GPS--overtwisted.

Theorem \ref{thm:MD2tight} can be extended to arbitrary neighborhoods of codimension 2 contact submanifolds $M$ whose normal bundle has a nowhere vanishing section:

\begin{theorem}\label{thm:nbdsubmanifoldtight}
Suppose $M$ is a contact submanifold of the contact manifold $(N, \xi)$. Assume that the normal bundle of $M$ has a nowhere vanishing section. Then, there is a neighborhood of $M$ in $N$ that is tight.
\end{theorem}

The proof of Theorems \ref{thm:MD2tight} and \ref{thm:nbdsubmanifoldtight} is based on Theorem \ref{thm:MD2ktight}. That theorem states that for $m$ large enough $(M \times P^{2m}(\varepsilon), \ker(\alpha + \sum_{i = 1}^m r^2_i\,d\theta_i))$ admits a contact embedding in a closed contact manifold of the same dimension that is Stein fillable, therefore $M \times P^{2m}(\varepsilon)$ is obviously tight. However, Theorems \ref{thm:MD2tight} and \ref{thm:nbdsubmanifoldtight} do not prove such a strong result. Their proof uses \cite[Theorem 1.1.(ii)]{CMP} and some packing lemmas to obtain a contradiction by stabilizing and reducing to Theorem \ref{thm:MD2ktight}.

\subsection{Applications}

\subsubsection{Remarks about contact submanifolds.} We are assuming a choice of contact forms whenever the measure of a radius of the tubular neighborhoods of a contact submanifold is required.

\smallskip

{\bf 1.}  Assume that $(M, \xi)$ contact embeds into an overtwisted contact manifold $(N, \xi_{OT})$ as a codimension $2$ submanifold with trivial normal bundle. By Theorem \ref{thm:MD2tight}, it is clear that the overtwisted disk cannot be localized on arbitrary small neighborhoods of $M$, even assuming that $M$ itself is overtwisted. This stands in sharp contrast with \cite{huang} and \cite{CMP} in which it is shown that the overtwisted disk can be localized around a very special kind of codimension $n$ submanifold: a plastikstufe \cite{niederkruger}.

{\bf 2.} Assume now that $(M, \xi_{OT})$ is overtwisted and contact embeds into a tight contact manifold $(N, \xi)$ as codimension $2$ submanifold with trivial normal bundle. Then we can perform a fibered connected sum of $(N, \xi)$ with itself along $(M, \xi_{OT})$. The gluing region is $M \times (-\varepsilon,\varepsilon) \times \mathbb S^1$, for some $\varepsilon>0$, and coordinates $(p,t, \theta)$ can be chosen such that the glued contact structure admits an associated contact form $\alpha= \alpha_{OT} +t\,d\theta$.

It is clear that the contact connection associated to the contact fibration
$M \times (-\varepsilon,\varepsilon) \times \mathbb S^1 \to (-\varepsilon,\varepsilon) \times \mathbb S^1$ \cite{presas2007class} induces the identity when we lift by parallel transport the loop $\{ 0 \} \times \mathbb S^1$. The parallel transport of an overtwisted disk of the fiber induces a plastikstufe, see \cite{presas2007class} for more details. By \cite{huang}, the manifold is overtwisted.
 
Call $R_M>0$ the biggest radius for which $M \times \mathbb D^2(R_M)$ contact embeds in $N$. The connected sum $N \#_M N$ readily increases the biggest radius to be $R_{N\#_M N} \geq \sqrt{2}R_M$: the annulus has twice the area of the original disk and therefore you can embed a disk of radius $\sqrt{2}R_M$. However we get much more, since we actually obtain $R_{N\#_M N}=\infty$. This is because we can always formally contact embed $M \times \R^2$ into $N \#_M N$. Moreover, we can assume that the embedding restricted to a very small neighborhood $U$ of the fiber $M \times \{ 0 \}$ provides a honest fibered contact embedding into $M \times (-\varepsilon,\varepsilon) \times \mathbb S^1$. Indeed, applying \cite[Corollary 1.4]{BEM} relative to the domain $U$ we obtain a contact embedding of $M \times \R^2$ thanks to the fact that $N \#_M N$ is overtwisted. This just means that the contact embedding of the tubular neighborhood can be really sophisticated and its explicit construction is far from obvious.


\subsubsection{Small loops of contactomorphisms.} 
Theorem \ref{thm:MD2tight} allows to extend the result of non--existence of small positive loops of contactomorphisms in overtwisted 3--manifolds contained in \cite{CPS} to arbitrary dimension. A loop of contactomorphisms or, more generally, a contact isotopy is said to be positive if it moves every point in a direction positively transverse to the contact distribution. The notion of positivity induces for certain manifolds, called orderable, a partial order on the universal cover of the contactomorphism group and it is related with non--squeezing and rigidity in contact geometry, see \cite{ek, ekp}. As explained in \cite{ek}, orderability is equivalent to the non--existence of a positive contractible loop of contactomorphisms.

Any contact isotopy is generated by a contact Hamiltonian $H_t \colon M \to \R$ that takes only positive values in case the isotopy is positive. The main result of \cite{CPS} states that if $(M, \ker\,\alpha)$ is an overtwisted 3--manifold there exists a constant $C(\alpha)$ such that any positive loop of contactomorphisms generated by a Hamiltonian $H \colon M \times S^1 \to \R^+$ satisfies $||H||_{\mathcal{C}^0} \ge C(\alpha)$. The result has been recently extended to arbitrary hypertight or Liouville (exact symplectically) fillable contact manifolds in \cite{albersmerry}. As a consequence of Theorem \ref{thm:MD2tight}, we can eliminate the restriction on the dimension in the overtwisted case:

\begin{theorem}\label{thm:nosmallpositiveloops}
Let $(M, \ker \,\alpha)$ be an overtwisted contact manifold. There exists a constant $C(\alpha)$ such that the norm of a Hamiltonian $H \colon M \times S^1 \to \R^+$ that generates a positive loop $\{\phi_{\theta}\}$ of contactomorphisms on $M$ satisfies
$$||H||_{\mathcal{C}^0} \ge C(\alpha)$$
\end{theorem}

The strategy of the proof copies that of \cite{CPS}. The first step is to prove that $M \times \mathbb D^2 (\varepsilon)$ is tight, this is provided by Theorem \ref{thm:MD2tight}. The second step shows that a small positive loop provides a way to lift a plastikstufe in $M$   (whose existence is equivalent to overtwistedness as discussed above \cite{huang}) to a plastikstufe in $M \times \mathbb D^2 (\varepsilon)$. This is exactly Proposition 9 in \cite{CPS}. This provides a contradiction that forbids the existence of the small positive loop.

It is worth mentioning that the argument forbids the existence of (possibly non--contractible) small positive loops. This is in contrast with \cite{albersmerry} and the work in progress by S. Sandon \cite{sandon2016} in which they need to add the contractibility hypothesis in order to conclude.

\begin{remark}
The hypothesis in Theorem \ref{thm:nosmallpositiveloops} can be changed by the probably weaker notion of $GPS$-overtwisted, see  
\cite{niederkrugerpresas}. Indeed, assume that the manifold $(M, \xi)$ is $GPS$-overtwisted. This means that there is an immersed $GPS$ in the manifold. The positive loop produce a $GPS$ in $M\times \mathbb D^2 (\varepsilon)$ by parallel transport of the $GPS$ around a closed loop in the base $D^2 (\varepsilon)$. In this case, we need to iterate the process $k$ times to produce a $GPS$ in $M\times P^{2k} (\varepsilon)$. Now, Theorem \ref{thm:MD2ktight} concludes that this manifold embeds into a Stein fillable one providing a contradiction with the main result in \cite{niederkrugerpresas}.
\end{remark}

\section*{Acknowledgements}
The authors express their gratitude to Roger Casals for the useful conversations around this article. The authors have been supported by the Spanish Research Projects SEV-2015-0554, MTM2015-63612-P, MTM2015-72876-EXP and MTM2016-79400-P. The second author was supported during the development of the article by a Master grant from ICMAT through the Severo Ochoa program.

\section{$M \times P^{2m}(\varepsilon)$ admits a Stein fillable smooth compactification}\label{sec:MD2k}

\subsection{Construction of a formal contact embedding $M \to \partial W$ with trivial normal bundle}

Recall that $(\xi, d\alpha)$ defines a symplectic vector bundle over $M$, thus it is equipped with a complex bundle structure unique up to homotopy. Denote $\xi^*$ the dual complex vector bundle of $\xi$.
A standard result on the theory of vector bundles guarantees the existence of a complex vector bundle $\tau \to M$ such that $\xi^* \oplus \tau \to M$ is trivial, that is, there is an isomorphism of complex vector bundles over $M$ between $\xi^* \oplus \tau$ and $M \times \mathbb{C}^k = \underline{\mathbb{C}}^k$, where $k$ is a positive integer large enough.

Denote $\pi\colon T^*M \to M$ the cotangent bundle projection and denote $\pr \colon \pi^*\tau \to T^*M$ the bundle projection. Define $\tilde{\pi} = \pi \circ \pr$.
Let us understand $\widehat{W} = \pi^*\tau$ as a smooth almost complex manifold. 
Choosing a $\xi$--compatible contact form $\alpha$, i.e $\xi = \ker\,\alpha$, it is clear that
\begin{align*}
T\widehat{W} &\cong\tilde{\pi}^*\tau \oplus \pr^*T(T^*M) \cong \tilde{\pi}^*\tau \oplus \tilde{\pi}^*T^*M \oplus \tilde{\pi}^*TM \cong \tilde{\pi}^*\tau \oplus \tilde{\pi}^*(\xi^* \oplus \langle \alpha \rangle) \oplus \tilde{\pi}^*TM \\
&\cong \tilde{\pi}^*(\tau \oplus \xi^*) \oplus \tilde{\pi}^* \langle \alpha \rangle \oplus \tilde{\pi}^*TM \cong \tilde{\pi}^*\underline{\mathbb{C}}^k \oplus \tilde{\pi}^* \langle \alpha \rangle \oplus \tilde{\pi}^*TM  \\
\end{align*}
In particular, the vector bundle $\pi^*\tau  \stackrel{\tilde \pi}{\to} M$ is isomorphic to
$\underline{\mathbb{C}}^k \oplus \langle \alpha \rangle$.
Fix a direct sum bundle metric $h$ in $\pi^*\tau$ such that $h(\alpha, \alpha) = 1$. Now define 

\centerline{$W = \{(v, p) \in \widehat{W}: h(v, v) \le 1\}$.}

Given a complex structure $j$ in $\xi$ compatible with $d\alpha$, we can extend it to a complex structure on $T^*M$ and by a direct sum with a complex structure in $\tau$ we obtain a complex structure $J$ in $T\widehat{W}$.
Then, $(W, J)$ is an almost complex manifold with boundary $\partial W$ that has a natural formal contact structure $\xi_0 = T\partial W \cap J(T \partial W)$. Consider the embedding

\centerline{$e_0 \colon M \to \partial W = \mathbb{S}(\underline{\mathbb{C}}^k \oplus \langle\alpha\rangle): p \mapsto (0,1)$.}

\noindent
We claim that its normal bundle is trivial because it is equal to $\tilde{\pi}^* \underline{\mathbb{C}}^k$. The reason is that the normal bundle to a section of a vector bundle is the restriction of the vertical bundle to the section.
In our case the restriction of the vertical bundle $T(\mathbb S(\underline{\mathbb{C}}^k) \oplus  \langle\alpha\rangle))$ to the image of $e_0$ is clearly $(\tilde \pi^* \underline{\mathbb{C}}^k)_{|im(e_0)}$.


\subsection{$W$ is Stein fillable and $\partial W$ is contact}

The distribution $T \partial W \cap J (T \partial W)$ is not necessarily a contact structure in $\partial W$. However, we will deform this distribution to a genuine contact structure using the following result.

\begin{theorem}[Eliashberg \cite{eliashberg90}]\label{thm:almostcomplex}
Let $(V^{2n}, J)$ be an almost complex manifold with boundary of dimension $2n > 4$ and suppose that $f \colon V \to [0,1]$ is a Morse function constant on $\partial V$ such that $\mathrm{ind}_p(f) \le n$ for every $p \in \mathrm{Crit}(f)$. Then, there exists a homotopy of almost complex structures $\{J_t\}_{t = 0}^1$ such that $J_0 = J$, $J_1$ is integrable and $f$ is $J_1$--convex.
\end{theorem}

We are clearly in the hypothesis since our manifold $W$ is almost complex, has dimension $2k+1+\mathrm{dim}\,M > 4$ (because $2k \ge \mathrm{dim\,}\xi = \mathrm{dim}\,M - 1$) and deformation retracts to $M$.

From Theorem \ref{thm:almostcomplex} we obtain a homotopy of almost complex structures $\{J_t\}$ in $W$ such that $J_0 = J$ and, $J_1$ is integrable. Moreover $(W, J_1)$ is a Stein domain and $\partial W$ inherits a contact structure given by $\xi_1 = J_1(T \partial W) \cap T \partial W$.
In fact, there is a homotopy of formal contact structures between $\xi_0$ and $\xi_1$ provided by $\xi_t = J_t(T \partial W) \cap T \partial W$.

\subsection{Properties of the embedding $e_0 \colon (M, \xi) \to (\partial W, \xi_1)$.}
Recall the following definition:
\begin{definition}
An embedding $e \colon (M_0, \xi_0, J_0) \to (M_1, \xi_1, J_1)$ is called \emph{formal contact} if there exists an homotopy of monomorphisms $\{\Psi_t \colon TM_0 \to TM_1\}_{t = 0}^1$ such that $\Psi_0 = de$, $\xi_0 = \Psi^{-1}_1(\xi_1)$ and $\Psi_1 \colon (\xi_0, J_0) \to (\xi_1, J_1)$ is complex.
\end{definition}

So far we have produced an embedding $e_0 \colon (M, \xi, j) \to (\partial W, \xi_0, J_0)$ that is formal contact with the constant homotopy equal to $de_0$. Indeed,
 $de_0^{-1}(\xi_0) = \xi$ and $de_0(\xi)$ is a complex subbundle of $\xi_0$. There is a family of complex isomorphisms $\Phi_t \colon \xi_0 \to \xi_t$ such that $\Phi_0 = \id$. Fix a Reeb vector field $R$ associated to $\xi$ and define $\widehat{R}_0 = de_0(R)$. Build a family $\{R_t\}$ of vector fields in $T\partial W$ satisfying $R_0|_{\im\, e_0} = \widehat{R}_0$ and $\langle R_t \rangle \oplus \xi_t = T\partial W$. We take a family of metrics $g_t$ in $\partial W$ defined in the following way: its restriction to $\xi_t$ is hermitic for the complex bundle $(\xi_t, J_t)$ and $R_t$ is unitary and orthogonal to $\xi_t$.

Extend $\Phi_t$ to an isomorphism of $T\partial W|_{\im\,e_0}$ in such a way that $\Phi_t(R_0) = R_t$. Define

\centerline{$E_t = \Phi_t \circ de_0 \colon TM \to T\partial W$.}

\noindent
The family $\{E_t\}_{t = 0}^1$ is composed of
bundle monomorphisms and clearly satisfies that $E_1^{-1}(\xi_t) = \xi$ and $E_1(\xi)$ is a complex subbundle of $\xi_1$. Therefore, $(e_0, E_t)$ is a formal contact embedding.

Define $\mathcal{N}_t = E_t(TM)^{\perp g_t}$ that is a bundle over $\im \, e_0$ which is complex by construction. $\mathcal{N}_0$ is isomorphic to $\underline{\mathbb{C}}^k$ and therefore all the bundles $\mathcal{N}_t$ are trivial complex bundles.

\subsection{Obtaining a contact embedding via h--principle}

The only missing piece to complete the puzzle is to prove that the embedding  $e_0$ can be made contact. 

Using $h$--principle it is possible to deform $(e_0, E_t)$ to a contact embedding thanks to the following theorem (cf. \cite[Theorem 12.3.1]{eliashbergmishachev}):

\begin{theorem}\label{thm:hprinciple}
Let $(e, E_t)$, $e \colon (M_0, \xi_0 = \ker \,\alpha_0) \to (M_1, \xi_1 = \ker \,\alpha_1)$, be a formal contact embedding between closed contact manifolds such that $\dim M_0 + 2 < \dim M_1$. Then, there exists a family of embeddings $\widetilde{e}_t \colon M_0 \to M_1$ such that:
\begin{itemize}
\item $\widetilde{e}_0 = e$ and $\widetilde{e}_1$ is contact,
\item $d\widetilde{e}_1$ is homotopic to $E_1$ through monomorphisms $G_t \colon TM_0 \to TM_1$, lifting the embeddings $\widetilde{e}_t$, such that $G_t(\xi_0) \subset \xi_1$ and the restrictions $G_t|_{\xi_0} \colon (\xi_0, d\alpha_0) \to (\xi_1, d\alpha_1)$ are symplectic.
\end{itemize}
\end{theorem}

Theorem \ref{thm:hprinciple} applied to $(e_0, E_t)$ provides a family of embeddings $\{e_t\}$ in which $e_1 \colon (M, \xi) \to (\partial W, \xi_1)$ is a contact embedding and a family of monomorphisms $G_t \colon TM \to T\partial W$ that lift $e_t$ such that $G_0 = E_1$, $G_1 = de_1$ and $G_t(\xi) \subset \xi_1$ is a complex subbundle. 

\begin{lemma}\label{lem:normaltrivial}
The normal bundle of $\im (e_1)$ in $(\partial W, \xi_1)$ is trivial.
\end{lemma}
\begin{proof}
Recall that $\mathcal{N}_1= E_1(TM)^{\perp g_1}= G_0(TM)^{\perp g_1}$ is a trivial complex vector bundle. Define, for $t\in [1,2]$, $\mathcal{N}_t = G_{t-1}(TM)^{\perp g_1}$. Clearly, $\mathcal{N}_2$ is the normal bundle of the contact embedding $e_1$. Since $\mathcal{N}_1$ is a trivial vector bundle so is $\mathcal{N}_2$.

\end{proof}

Denote the $2m$--dimensional polydisk by $P^{2m}(r_1, \ldots, r_m) = \mathbb{D}^2(r_1) \times \cdots \times \mathbb{D}^2(r_m)$ and abbreviate it as $P^{2m}(r)$ when $r_1 = \ldots = r_m = r$. The following result summarizes the work completed in this section and an important consequence (namely, the title of the section): $M \times P^{2m}(\varepsilon)$ admits a smooth compactification into a Stein fillable contact manifold.

\begin{theorem}\label{thm:MD2ktight}
Any closed contact manifold $(M, \ker \, \alpha)$ contact embeds in the boundary of a Stein fillable manifold with trivial normal bundle.
Furthermore, there exists $k \ge 1$ such that for any $m\geq k$
$$
\left(M \times P^{2m}(\varepsilon), \ker\left(\alpha + \sum_{i = 1}^m r^2_i\,d\theta_i \right)\right)
$$
is tight with $\varepsilon > 0$ small enough depending only on $\alpha$ and $k$.
\end{theorem}
\begin{proof}
The map $e_1$ proves the first part because by Lemma \ref{lem:normaltrivial} the normal bundle of the contact embedding $e_1 \colon (M, \xi) \to (\partial W, \xi_1)$ is trivial. Notice that the codimension of the embedding is equal to $2k = \mathrm{dim}\,\tau$ and by replacing $\tau$ with $\tau' = \tau \oplus \underline{\mathbb{C}}^{m-k}$ we obtain embeddings of arbitrary codimension $2m \ge 2k$.

Suppose henceforth that $m \ge k$.
By an standard neighborhood theorem in contact geometry it follows that there is a contactomorphism between a neighborhood of $\im (e_1)$ in $(\partial W, \xi_1)$ and a neighborhood of $M \times \{0\}$ in $(M \times \R^{2m}, \ker(\alpha + \sum_{i = 1}^k r_i^2 d\theta_i))$. Therefore, for some $\varepsilon_0 > 0$, the previous contactomorphism provides an embedding from $M \times P^{2m}(\varepsilon_0)$ into $\partial W$.

Finally, since $(\partial W, \xi_1)$ is Stein filable, it is tight. Thus, any of its open subsets is also tight and the conclusion follows.
\end{proof}

\section{$M \times \mathbb{D}^2(\varepsilon)$ is tight if $\varepsilon$ is small}

The argument leading to Theorem \ref{thm:MD2ktight} provided no bound on the first positive integer $k$ such that $M \times P^{2k}(\varepsilon, \ldots, \varepsilon)$ is tight. Indeed, $k$ was fixed at the beginning of Section \ref{sec:MD2k}, depending on the rank of $\tau \to M$, the bundle constructed to make the sum $\xi^* \oplus \tau$ trivial.

The insight needed to prove Theorem \ref{thm:MD2tight} is supplied by the understanding of overtwisted contact manifolds briefly discussed in the introduction. To be more concrete, the precise statement we will use in this section, extracted from \cite{CMP}, is the following:

\begin{theorem}\label{thm:CMP}
Suppose that $(M, \ker \, \alpha)$ is an overtwisted contact manifold. Then, if $R$ is large enough, $(M \times \mathbb{D}^2(R), \ker(\alpha + r^2d\theta))$ is also overtwisted.
\end{theorem}


The idea is to embed $\partial W \times \mathbb{D}^2(R)$ in the boundary $\partial V$ of a Weinstein manifold. Using the embedding constructed in the previous section we obtain then an embedding $M \times \mathbb{D}^2(R) \rightarrow \partial V$ that has trivial normal bundle. This leads to the proof of a statement similar to Theorem \ref{thm:MD2ktight} in which we replace $(M, \ker \,\alpha)$ by $(M \times \mathbb{D}^2(R), \ker (\alpha + r^2d\theta))$.
Note that it is key to make sure that $R$ is arbitrarily large.

A Weinstein manifold $(W, \omega, f, Y)$ is a manifold with boundary $W$ equipped with a symplectic structure $\omega$, a Morse function $f \colon W \to \R$ and a Liouville vector field $Y$ that is a pseudo--gradient for $f$. Notice that the symplectic form is automatically exact, $\omega = L_Y \omega = d \,i_Y \omega$, so the boundary of a Weinstein manifold is exact symplectically fillable.

The product of Weinstein manifolds $(W_1, \omega_1, g_1, Y_1)$ and $(W_2, \omega_2, g_2, Y_2)$ can be equipped with a Weinstein structure. Indeed, define $\omega' = \omega_1 + \omega_2$ and $Y' = Y_1 + Y_2$. Clearly, $Y'$ is Liouville for $\omega'$. Suppose for simplicity that $g_1$ and $g_2$ are strictly positive (a rescaling would make the argument work in general) and define a function on $W = W_1 \times W_2$ by
$$
f_q = (g_1^q + g_2^q)^{1/q}
$$
for an arbitrary $q > 1$. It is easy to check that $\mathrm{Crit}(f_q) = \mathrm{Crit}(g_1) \times \mathrm{Crit}(g_2)$, the function $f_q$ is Morse and $Y'$ is pseudogradient for $f_q$.

The Stein fillable manifold $W$ supplied by Theorem \ref{thm:MD2ktight} is naturally equipped with a Weinstein structure $(W = f^{-1}(0,1], \omega, f, Y)$ that satisfies $\xi_1 = \ker(i_Y\omega|_{\partial W})$.
By the preceeding discussion,
a Weinstein structure in $W \times \R^2$ is given by $\omega + dx \wedge dy$, $X = Y + r\frac{\partial}{\partial r}$ and
$$
f_q = \left(f^q + \left(\frac{x^2+y^2}{(2R)^2}\right)^q \right)^{1/q}
$$
The critical points of $f_q$ have the form $(p, 0, 0)$, where $p \in \mathrm{Crit}(f)$.

The Liouville vector field $X$ is transverse to $\partial W \times \R^2$. Our aim is now to embed $\partial W \times \mathbb{D}^2(R)$ into a level set of $f_q$ by following $\phi_t$, the flow of $X$. We can easily show:

\begin{proposition}\label{prop:yellow}
For any $\delta>0$, there exists $q > 1$ large enough and a function $\mu\colon \partial W  \times \mathbb{D}^2(R) \to \R^-$ such that $||\mu||_{\mathcal C^0} \leq \delta$ and $\phi_{\mu}\colon \partial W  \times \mathbb{D}^2(R) \to W \times \mathbb{D}^2(R)$ satisfies $\phi_{\mu}(\partial W  \times \mathbb{D}^2(R)) \subset f_q^{-1}(1)$.
\end{proposition}

\begin{proof}
For $q \to \infty$,  the level set $f_q^{-1}(1)$ gets $C^{\infty}$--close to the submanifold $\partial W  \times \mathbb{D}^2(R)$. Since $X$ is transverse to both of them, the result follows.
\end{proof}

\begin{figure}[ht]
\includegraphics[scale=0.8]{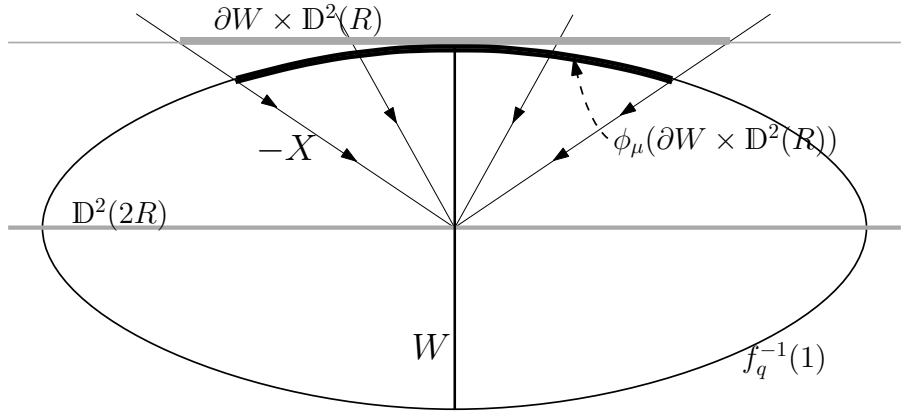}
\caption{Contact embedding of $\partial W  \times \mathbb{D}^2(R)$ into $f_q^{-1}(1)$.}\label{fig:embed}
\end{figure}

Proposition \ref{prop:yellow} produces a contactomorphism as the next lemma states.

\begin{lemma} \label{lem:trans}
Let $e\colon H \hookrightarrow M$ be a hypersurface transverse to a nowhere vanishing Liouville vector field $X$ in $(M, \omega)$, the $1$--form $e^* i_X \omega$ defines a contact structure on $H$. Moreover, if $\phi_t$ denotes the Liouville flow starting at $H$ and $s:H \to \R$ is a fixed function, then $\phi_{s} \circ e:H \hookrightarrow M$ is contactomorphic to $e$ provided the flow $\phi_s$ is well--defined.
\end{lemma}

Notice that the level set $f_q^{-1}(1)$ is the boundary of the Weinstein manifold $V = f_q^{-1}(0,1]$.
Denote $\alpha' = i_X (\alpha + dx \wedge dy)$.
A straightforward application of Lemma \ref{lem:trans} concludes the following:

\begin{proposition}\label{prop:WR2intoWeinstein}
For any $R > 0$, the contact manifold $(\partial W \times \mathbb{D}^2(R), \ker (\alpha'|_{\partial W \times \mathbb D^2(R)}))$ admits a contact embedding into the boundary of a Weinstein manifold.
\end{proposition}

Combining the last proposition and the results from the previous section we obtain:

\begin{corollary}\label{coro:fideoistight}
Given a contact manifold $(M, \alpha)$ there exists $k \in \mathbb{N}$ and $\varepsilon_0 > 0$ such that for every $R > 0$ the contact manifold
$(M \times P^{2k+2}(\varepsilon_0, \ldots, \varepsilon_0, R), \ker(\alpha + \sum _{i = 1}^{k+1}r_i^2 d\theta_i))$ is tight.
\end{corollary}
Let us emphasize that $\varepsilon_0$ does not depend on $R$: for any $R>0$, $M \times P^{2k+2}(\varepsilon_0, \ldots, \varepsilon_0, R)$ is tight.

\begin{proof}
The integer $k$ and the number $\varepsilon_0$ both come from Theorem \ref{thm:MD2ktight}.
Denote by $e'$ the contact embedding from $(M \times P^{2k}(\varepsilon_0), \ker(\alpha + \sum _{i = 1}^{k}r_i^2 d\theta_i))$ into $(\partial W, \xi_1 = \ker(i_Y \omega))$ and let $\eta$ be the conformal factor of $e'$, $(e')^* i_Y \omega = {\rm exp}(\eta) \alpha'$.
If necessary, decrease the value of $\varepsilon_0$ to guarantee that ${\rm sup}\,\eta$ is finite.

Proposition \ref{prop:WR2intoWeinstein} supplies a Weinstein manifold $(V = f_q^{-1}(0,1], \omega + dx \wedge dy, f_q, X)$ and  contact embedding
$$\varphi\colon (\partial W \times \mathbb{D}^2({\rm exp}({\rm sup}\, \eta/2) R), \alpha') \hookrightarrow (\partial V, \alpha').$$

Therefore, the map $\widetilde{\varphi} \colon M \times P^{2k+2}(\varepsilon_0, \ldots, \varepsilon_0, R) \to \partial V$ given by
$$
\widetilde{\varphi}(p, x, y, x_{k+1}, y_{k+1}) = \varphi(e'(p,x,y), {\rm exp}(\eta/2) x_{k+1}, {\rm exp}(\eta/2) y_{k+1}))
$$
is a contact embedding. Since $\partial V$ is exact symplectically fillable the conclusion follows.
\end{proof}

We are ready now to prove Theorem \ref{thm:MD2tight}. To ease the notation, we shall understand the contact form is equal to $\alpha + \sum r_i^2 d \theta_i$ in case it is omitted.

Let us proceed by contradiction. Suppose that $M \times \mathbb{D}^2(\varepsilon)$ is overtwisted for $\varepsilon$ smaller than $\varepsilon_0$. Applying Theorem \ref{thm:CMP} $k$ times consecutively we obtain a radius $R_{\varepsilon} > 0$ such that $M \times P^{2k+2}(\varepsilon, R_{\varepsilon}, \ldots, R_{\varepsilon})$ is overtwisted.
As we will show below, this manifold contact embeds into $M \times P^{2k+2}(\varepsilon_0, \ldots, \varepsilon_0, R)$ provided $R$ is large enough. From Corollary \ref{coro:fideoistight} we know that the latter manifold is tight so we reach a contradiction. Therefore, $M \times \mathbb{D}^2(\varepsilon)$ is tight.

The only missing ingredient is the announced contact embedding:
\begin{equation}\label{eq:embedding}
M \times P^{2k+2}(\varepsilon, R_{\varepsilon}, \ldots, R_{\varepsilon}) \to
M \times P^{2k+2}(\varepsilon_0, \ldots, \varepsilon_0, R)
\end{equation}
Its existence, subject to the conditions $\varepsilon < \varepsilon_0$ and $R$ large enough,
 is a consequence of the following packing theorem in symplectic geometry proved by Guth \cite[Theorem 1]{guth}.

\begin{theorem}\label{thm:guth}
For every $m \in \mathbb{N}$ there is a constant $C(m) \ge 1$ such that for any pair of ordered $m$--tuples of positive numbers $R_1 \le \ldots \le R_m$ and $R'_1 \le \ldots \le R'_m$ that satisfy
\begin{itemize}
\item $C(m) R_1 \le R'_1$ and
\item $C(m) R_1 \cdot \ldots \cdot R_k \le R'_1 \cdot \ldots \cdot R'_m$.
\end{itemize}
there is a symplectic embedding
\[
P^{2m}(R_1, \ldots, R_m) \hookrightarrow P^{2m}(R'_1, \ldots, R'_m)
\]
\end{theorem}

The symplectic embedding supplied by Theorem \ref{thm:guth} is automatically extended to our desired contact embedding (\ref{eq:embedding}) thanks to the following lemma:

\begin{lemma}
Let $\Psi \colon (D_1, d\lambda_1) \to (D_2, d\lambda_2)$ be an exact symplectic embedding. For any contact manifold $(M, \ker \,\alpha)$ with a choice of contact form $\alpha$ that makes the associated Reeb flow complete, $\Psi$ induces a (strict) contact embedding
$$(M \times D_1, \alpha + \lambda_1) \rightarrow (M \times D_2, \alpha + \lambda_2).$$
\end{lemma}
\begin{proof}
Since $\Psi$ is exact, there exists a smooth function $H \colon D_1 \to \R$ such that $dH = \Psi^*\lambda_2 - \lambda_1$. If we denote the Reeb flow in $M$ by $\Phi$,
$$
\varphi \colon (M \times D_1, \alpha + \lambda_1) \to (M \times D_2, \alpha + \lambda_2), \enskip \enskip \varphi(p, x) = (\Phi_{-H(x)}(p), \Psi(x))
$$
is a contact embedding.
\end{proof}


\section{Extension to contact submanifolds}

The results from the previous sections can be extended to a more general setting: contact submanifolds with arbitrary normal bundle. In the presence of a nowhere vanishing section of the normal bundle we will prove that the contact submanifold has a tight neighborhood. This is the content of Theorem \ref{thm:nbdsubmanifoldtight}.

Let $\pi \colon E \to M$ be a complex vector bundle over a contact manifold equipped with an hermitian metric and a unitary connection $\nabla$. The associated vertical bundle is denoted by $\mathcal{V} = \ker(d\pi)$.
The standard Liouville form in $\R^{2n}$ is $U(n)$--invariant and induces a global 1--form in $\mathcal{V}$ that will be denoted $\widetilde{\lambda}$. This real 1--form can be extended to $TE$ by the expression $\lambda = \widetilde{\lambda} \circ \pi_{\mathcal{V}}$ after we choose a projection onto the vertical direction $\pi_{\mathcal{V}} \colon TE \to \mathcal{V}$. The map $\pi_V$ is determined by the choice of unitary connection so it is not canonical. The 1--form in $TE$ associated to the connection $\nabla$ is $\widetilde{\alpha} = \pi^*\alpha + \lambda$.

Even though $\widetilde{\alpha}$ can be seen as the lift of the contact form $\alpha$ to $E$, it is not a globally defined contact form in general. However, it defines a contact form around the zero section $E_0$ of the vector bundle.

\begin{lemma}\label{lem:aroundzeroiscontact}
$\widetilde{\alpha}$ is a contact form in a neighborhood of $E_0$. The restriction $(E_0, \ker (\widetilde{\alpha}|_{E_0}))$ is contactomorphic to $(M, \xi = \ker\, \alpha)$. Moreover,
given any other contact structure $\ker \, \beta$ that coincides with $\ker (\widetilde{\alpha})$ in $E_0$ and with the same complex structure in the normal bundle, there exist neighborhoods $U, V$ of $E_0$ such that $(U, \ker(\beta|_U))$ and $(V, \ker(\widetilde{\alpha}|_V))$ are contactomorphic.

\end{lemma}

Suppose henceforth that $\pi$ has a global nowhere vanishing section $s \colon M \to E$. The section $s$ creates a complex line subbundle $\pi|_L \colon L \to M$.
Then, the bundle $E$ splits as $E = F \oplus L$ and $L$ is trivial, i.e. there is an isomorphism $\phi \colon L \to \underline{\mathbb{C}}$ that sends $s(p)$ to $1_p \in \mathbb{C}$ in the fiber above every point $p \in M$.

A suitable choice of unitary connection on $\pi \colon E \to M$ ensures that the associated contact form can be written as $\widetilde{\alpha} = \alpha' + \lambda$, where $\alpha'$ is a contact form in $F$ and $\lambda$ is the radial Liouville form in $\R^2$.

\begin{proposition}\label{prop:UD2}
There exists $U$, a neighborhood of the zero section $F_0$ of $F$, and $\varepsilon > 0$ such that
$(U \times \mathbb D^2(\varepsilon), \ker(\alpha' + \lambda))$ is tight.
\end{proposition}

Note that this statement is exactly Theorem \ref{thm:MD2tight} except from the fact that $F$ is not closed. The proof of Proposition \ref{prop:UD2} follows by embedding $(U, \ker\, \alpha')$ in a closed contact manifold $(\widetilde{F}, \ker\,\widetilde{\alpha}')$ 
 and then applying Theorem \ref{thm:MD2tight} to this manifold to deduce that $(\widetilde{F} \times \mathbb D^2(\varepsilon), \ker(\widetilde{\alpha}' + \lambda))$ is tight if $\varepsilon > 0$ is small. This result evidently implies that $(U \times \mathbb D^2(\varepsilon), \ker(\alpha' + \lambda))$ is also tight.

The aforementioned embedding is defined by the natural inclusion of $F$ in the projectivization of $F \oplus \mathbb{C}$:
$$
F \hookrightarrow Q = \mathbb{P}(F \oplus \mathbb{C})
$$
The complex bundle $\pi_Q\colon Q \to M$ carries a natural formal contact structure 
$\xi' = (d\pi_Q)^{-1}(\xi)$
Indeed, an almost complex structure in $\xi'$ is obtained as the sum of the pullback of a complex structure in $\xi$ compatible with $d\alpha$ and a complex structure on the fibers of $\pi_Q$. This formal contact structure is genuine (i.e., it is a true contact structure) in a neighborhood $U$ of $F_0$ by Lemma \ref{lem:aroundzeroiscontact}. The $h$--principle for closed manifolds proved in \cite[Theorem 1.1]{BEM} provides a homotopy from any formal contact structure to a contact structure. Furthermore, the homotopy can be made relative to a closed set in which the formal contact structure is already genuine. Applying this theorem we obtain a contact structure $\widetilde{\xi}'$ on $Q$ that agrees with $\ker\, \alpha'$ in $U$.

We can reformulate Proposition \ref{prop:UD2} in the following way:

\begin{theorem}\label{thm:vectorbundlewithsection}
Let $\pi \colon E \to M$ be a complex vector bundle over a closed contact manifold $(M, \xi)$. Suppose that $\pi$ has a global nowhere vanishing section. Then, there exists a neighborhood $U$ of the zero section of the bundle such that $(U, \widetilde \xi)$ is tight for any contact structure $\widetilde{\xi}$ extending $\xi$ and preserving the complex structure of $E$.
\end{theorem}

An immediate application of Theorem \ref{thm:vectorbundlewithsection} to the case in which $M$ is a contact submanifold and $\pi$ is its normal bundle yields Theorem \ref{thm:nbdsubmanifoldtight}.


\begin{thebibliography}{XXXX}

\bibitem{albersmerry}
Peter Albers, Urs Fuchs, and Will J. Merry. Positive loops and $L^{\infty}$--contact systolic inequalities. \emph{Selecta Math.} (N.S.), 23(4):2491--2521, 2017.

\bibitem{BEM}
Matthew Strom Borman, Yakov Eliashberg, and Emmy Murphy. Existence and classification of overtwisted
contact structures in all dimensions.\emph{ Acta Math.}, 215(2):281--361, 2015.

\bibitem{CMP}
R. Casals, E. Murphy, and F. Presas. Geometric criteria for overtwistedness. \emph{ArXiv e-prints}, March 2015.

\bibitem{CPS}
 Roger Casals, Francisco Presas, and Sheila Sandon. Small positive loops on overtwisted manifolds. \emph{J. Symplectic Geom.}, 14(4):1013--1031, 2016.
 
\bibitem{eliashbergmishachev}
 Y. Eliashberg and N. Mishachev. \emph{Introduction to the h-principle}, volume 48 of \emph{Graduate Studies in Mathematics}. American Mathematical Society, Providence, RI, 2002.

\bibitem{ek}
Y. Eliashberg and L. Polterovich. Partially ordered groups and geometry of contact transformations. Geom.
\emph{Funct. Anal.}, 10(6):1448--1476, 2000.

\bibitem{eliashberg90}
Yakov Eliashberg. Topological characterization of Stein manifolds of dimension $> 2$. \emph{Internat. J. Math.}, 1(1):29--46, 1990.


\bibitem{ekp}
Yakov Eliashberg, Sang Seon Kim, and Leonid Polterovich. Geometry of contact transformations and domains:
orderability versus squeezing. \emph{Geom. Topol.}, 10:1635--1747, 2006.


\bibitem{gironella}
F. Gironella. On some examples and constructions of contact manifolds.\emph{ ArXiv e-prints}, November 2017.


\bibitem{guth}
Larry Guth. Symplectic embeddings of polydisks. \emph{Invent. Math.}, 172(3):477--489, 2008.


\bibitem{huang}
Yang Huang. On plastikstufe, bordered Legendrian open book and overtwisted contact structures. \emph{J. Topol.}, 10(3): 720--743, 2017.


\bibitem{niederkruger}
Klaus Niederkr\"uger. The plastikstufe|a generalization of the overtwisted disk to higher dimensions. \emph{Algebr. Geom. Topol.}, 6:2473--2508, 2006.

\bibitem{niederkrugerpresas} Klaus Niederkr\"uger and Francisco Presas. Some remarks on the size of tubular neighborhoods in contact topology and fillability. \emph{Geom. Topol.}, 14(2):719--754, 2010.

\bibitem{presas2007class}
Francisco Presas et al. A class of non--fillable contact structures. \emph{Geom. Topol.}, 11(4): 2203--2225, 2007.

\bibitem{sandon2016}
Sheila Sandon. Floer homology for translated points. 2016.

\end{thebibliography}

\end{document}